\documentclass[11pt]{amsart}

\usepackage[T1]{fontenc}
\usepackage[utf8]{inputenc}
\usepackage{lmodern}
\usepackage{microtype}
\usepackage{amsmath,amssymb,amsthm,mathtools}
\usepackage{enumitem}
\usepackage{booktabs}
\usepackage[hidelinks]{hyperref}
\usepackage[nameinlink,capitalize,noabbrev]{cleveref}

\newtheorem{theorem}{Theorem}[section]
\newtheorem{proposition}[theorem]{Proposition}
\newtheorem{lemma}[theorem]{Lemma}
\newtheorem{corollary}[theorem]{Corollary}
\newtheorem{conjecture}{Conjecture}

\crefname{theorem}{theorem}{theorems}
\Crefname{theorem}{Theorem}{Theorems}
\crefname{proposition}{proposition}{propositions}
\Crefname{proposition}{Proposition}{Propositions}
\crefname{lemma}{lemma}{lemmas}
\Crefname{lemma}{Lemma}{Lemmas}
\crefname{corollary}{corollary}{corollaries}
\Crefname{corollary}{Corollary}{Corollaries}
\crefname{conjecture}{conjecture}{conjectures}
\Crefname{conjecture}{Conjecture}{Conjectures}
\theoremstyle{definition}

\theoremstyle{remark}
\newtheorem{remark}[theorem]{Remark}
\crefname{definition}{definition}{definitions}
\Crefname{definition}{Definition}{Definitions}
\crefname{remark}{remark}{remarks}
\Crefname{remark}{Remark}{Remarks}

\newcommand{\PP}{\mathbb P}
\newcommand{\CC}{\mathbb C}
\newcommand{\QQ}{\mathbb Q}
\newcommand{\OO}{\mathcal O}
\newcommand{\Ric}{\operatorname{Ric}}
\newcommand{\rank}{\operatorname{rank}}
\newcommand{\Aut}{\operatorname{Aut}}
\newcommand{\Sym}{\operatorname{Sym}}
\newcommand{\FS}{\mathrm{FS}}
\newcommand{\divisor}{\operatorname{div}}

\title[Projectively induced K\"ahler--Einstein surfaces]
{Compact Projectively Induced K\"ahler--Einstein Surfaces}

\author{Mirel Caib{\u a}r}
\address{(Mirel Caib{\u a}r) Department of Mathematics\\
The Ohio State University\\
1760 University Drive, Mansfield, OH 44906, USA}
\email{caibar@math.ohio-state.edu}

\author{Andrea Loi}
\address{(Andrea Loi) Dipartimento di Matematica \\
Universit\`a di Cagliari, Via Ospedale 72, 09124 Cagliari (Italy)}
\email{loi@unica.it}


\begin{document}

\begin{abstract}
We classify compact K\"ahler--Einstein surfaces whose metric is induced by a
holomorphic isometric immersion into a finite-dimensional complex projective
space. No symmetry assumption and no bound on the codimension are imposed. We
prove that the only such surfaces are
\[
(\PP^2,m g_{\FS})
\quad\text{and}\quad
\bigl(\PP^1\times\PP^1,m(g_{\FS}\oplus g_{\FS})\bigr),
\qquad m\in\mathbb Z_{>0},
\]
realized respectively by the Veronese and Segre--Veronese embeddings.

The main new ingredient is a codimension-independent exclusion of the entire
Fano-index-one branch, combining a common anticanonical root construction with
Gram--Gauss rank estimates and, in degree five, an equivariant curvature
argument.
Consequently, every connected compact K\"ahler--Einstein surface whose metric
is induced by a holomorphic isometric immersion into a finite-dimensional
complex projective space is homogeneous.
\end{abstract}

\maketitle

\section{Introduction}

A classical problem in K\"ahler geometry asks which K\"ahler manifolds can be
holomorphically and isometrically immersed into a complex space form. In the
positive-curvature case, one considers a K\"ahler manifold $(X,g)$ admitting a
holomorphic isometric immersion
\[
\varphi:(X,g)\longrightarrow (\PP^N,g_{\FS}),
\qquad N<\infty,
\]
and says that $g$ is \emph{projectively induced}.

A conjecture of Loi and Zedda
\cite[Conjecture~4.2]{LoiZedda2018} predicts that a complete
K\"ahler--Einstein manifold admitting a K\"ahler immersion into a
finite-dimensional complex projective space should be homogeneous. We shall
use the following compact form of their conjecture.

\begin{conjecture}[Loi--Zedda]\label{conj:homogeneity}
Let $(X,g)$ be a compact K\"ahler--Einstein manifold. If $g$ is induced by a
holomorphic isometric immersion into a finite-dimensional complex projective
space, then $(X,g)$ is homogeneous.
\end{conjecture}

Related homogeneity questions for projectively induced K\"ahler metrics were
studied by Loi, Salis and Zuddas \cite{LoiSalisZuddas2021}.

Hulin proved a fundamental first restriction: the Einstein constant of a
compact projectively induced K\"ahler--Einstein manifold is necessarily
positive \cite[Theorem]{Hulin2000}. Thus every compact example is Fano. Earlier
classification results in small codimension, beginning with work of Smyth and
Chern on hypersurfaces and continuing with Tsukada in codimension two, already
exhibited strong rigidity
\cite{Smyth1967,Chern1967,Tsukada1986}. Several further classification
results under symmetry assumptions were obtained by Manno and Salis
\cite{MannoSalis2022,MannoSalis2026}, while Di Scala and Sombra developed
complementary algebraic methods for toric Fano manifolds
\cite{DiScalaSombra2025}.

Very recently, Huang posted a preprint whose Theorem~B states the compact toric
classification in arbitrary complex dimension, and hence the toric case of the
Loi--Zedda conjecture \cite{Huang2026}. This result is independent of the
present argument and is not used below.  The two results are complementary:
\cite{Huang2026} treats the toric case in all dimensions,
whereas the del Pezzo surfaces of degrees one through five, precisely the
surfaces excluded in Step~3 below, are not toric, so none of the exclusions
obtained here follows from it.

The purpose of this paper is to remove all symmetry assumptions in complex
dimension two.

\begin{theorem}[Compact classification]\label{thm:main}
Let $(X^2,g)$ be a connected compact K\"ahler--Einstein surface and let
\[
\varphi:(X^2,g)\longrightarrow(\PP^N,g_{\FS}),
\qquad N<\infty,
\]
be a holomorphic isometry. Then $(X,g)$ is homogeneous.

More precisely, after replacing $\PP^N$ by the projective linear
span of $\varphi(X)$, there exists an integer $m\ge 1$ such that, up to
a biholomorphic isometry of the source and a projective-unitary
transformation of the target, exactly one of the following occurs:
\begin{enumerate}[label=\textup{(\roman*)}]
\item
\[
(X,g)\simeq \bigl(\PP^2,m g_{\FS}\bigr),
\]
and $\varphi$ is projectively-unitarily congruent to the standard
\emph{isometric $m$-th Veronese embedding}
\[
\nu_m:
\bigl(\PP^2,m g_{\FS}\bigr)
\longrightarrow
\left(
\PP^{\binom{m+2}{2}-1},
g_{\FS}
\right),
\]
which, in homogeneous coordinates $[Z_0:Z_1:Z_2]$, is given by
\[
\nu_m([Z_0:Z_1:Z_2])
=
\left[
\sqrt{\frac{m!}{\alpha_0!\alpha_1!\alpha_2!}}\,
Z_0^{\alpha_0}Z_1^{\alpha_1}Z_2^{\alpha_2}
\right]_{\alpha_0+\alpha_1+\alpha_2=m}.
\]
Here the target carries its standard Hermitian structure.  The
multinomial normalization is characterized by the identity
\[
\sum_{\alpha_0+\alpha_1+\alpha_2=m}
\frac{m!}{\alpha_0!\alpha_1!\alpha_2!}
|Z_0|^{2\alpha_0}|Z_1|^{2\alpha_1}|Z_2|^{2\alpha_2}
=
\bigl(|Z_0|^2+|Z_1|^2+|Z_2|^2\bigr)^m,
\]
and consequently
\[
\nu_m^*g_{\FS}=m g_{\FS}.
\]
In particular, if $\varphi$ is linearly full, then
\[
N=\binom{m+2}{2}-1.
\]

\item
\[
(X,g)\simeq
\left(
\PP^1\times\PP^1,
m(g_{\FS}\oplus g_{\FS})
\right),
\]
and $\varphi$ is projectively-unitarily congruent to the standard
\emph{isometric $(m,m)$ Segre--Veronese embedding}
\[
\operatorname{SV}_{m,m}:
\left(
\PP^1\times\PP^1,
m(g_{\FS}\oplus g_{\FS})
\right)
\longrightarrow
\left(
\PP^{(m+1)^2-1},
g_{\FS}
\right)
\]
associated with $\OO_{\PP^1\times\PP^1}(m,m)$.  In homogeneous coordinates
\[
([Z_0:Z_1],[W_0:W_1])
\in
\PP^1\times\PP^1,
\]
it is given by
\[
\operatorname{SV}_{m,m}
\bigl([Z_0:Z_1],[W_0:W_1]\bigr)
=
\left[
\sqrt{\binom{m}{a}\binom{m}{b}}\,
Z_0^{m-a}Z_1^a
W_0^{m-b}W_1^b
\right]_{0\le a,b\le m}.
\]
The binomial normalization gives
\[
\begin{aligned}
&\sum_{a,b=0}^m
\binom{m}{a}\binom{m}{b}
|Z_0|^{2(m-a)}|Z_1|^{2a}
|W_0|^{2(m-b)}|W_1|^{2b}
\\
&\qquad =
\bigl(|Z_0|^2+|Z_1|^2\bigr)^m
\bigl(|W_0|^2+|W_1|^2\bigr)^m,
\end{aligned}
\]
and hence
\[
\operatorname{SV}_{m,m}^*g_{\FS}
=
m(g_{\FS}\oplus g_{\FS}).
\]
In particular, if $\varphi$ is linearly full, then
\[
N=(m+1)^2-1.
\]
\end{enumerate}
\end{theorem}

Thus \cref{thm:main} proves Conjecture~A in complex dimension two, namely the
compact two-dimensional case of the Loi--Zedda conjecture.

The point is not merely to recover the toric classification. Once positivity
reduces the problem to del Pezzo surfaces, there remain many non-toric
K\"ahler--Einstein examples. The smooth del Pezzo surfaces are
$\PP^1\times\PP^1$ and the blow-ups of $\PP^2$ at at most eight points in
general position. Their K\"ahler--Einstein existence problem is classical:
all smooth del Pezzo surfaces admit a K\"ahler--Einstein metric except the
blow-ups of $\PP^2$ at one or two points; see Tian \cite{Tian1990} and the
later moduli-theoretic treatment \cite{OdakaSpottiSun2016}.

For the Fano-index-one branch, writing
\[
d=(-K_X)^2,
\]
the landscape relevant to the proof is therefore
\begin{center}
\begin{tabular}{cclcc}
\toprule
$d$ & blow-ups & anticanonical model & KE? & projectively induced KE?\\
\midrule
8 & $\operatorname{Bl}_{1}\PP^2$ & --- & no & no\\
7 & $\operatorname{Bl}_{2}\PP^2$ & --- & no & no\\
6 & $\operatorname{Bl}_{3}\PP^2$ & toric del Pezzo & yes & no\\
5 & $\operatorname{Bl}_{4}\PP^2$ & $Y_5\subset\PP^5$ & yes & no\\
4 & $\operatorname{Bl}_{5}\PP^2$ & two quadrics in $\PP^4$ & yes & no\\
3 & $\operatorname{Bl}_{6}\PP^2$ & cubic in $\PP^3$ & yes & no\\
2 & $\operatorname{Bl}_{7}\PP^2$ & double plane & yes & no\\
1 & $\operatorname{Bl}_{8}\PP^2$ & anticanonical pencil & yes & no\\
\bottomrule
\end{tabular}
\end{center}

The last column summarizes the index-one exclusion. The genuinely new
exclusions in this paper are the degrees $1$ through $5$. Degree $6$ is
already covered by previously known toric rigidity, while degrees $7$ and
$8$ are excluded by the classical nonexistence of K\"ahler--Einstein
metrics. In particular, the theorem rules out positive-dimensional moduli
families of non-toric K\"ahler--Einstein surfaces, including smooth cubic
surfaces and degree-four del Pezzo surfaces. This is precisely the non-toric
regime not covered by the symmetry-based classification results recalled
above.

The proof is organized in five steps. First, positivity of the Einstein
constant reduces the problem to smooth del Pezzo surfaces. Second, the Ricci-class
identity places the projective polarization on the primitive Fano ray and
reduces the problem to the three possible Fano indices. Third, the entire
index-one branch is excluded: a common anticanonical root removes degrees one
and two; a Gram--Gauss rank comparison removes degrees three and four; and the
degree-five surface is ruled out by its $S_5$-equivariant anticanonical
representation together with an explicit curvature computation. Degrees six
through eight are excluded by previously known toric rigidity and
K\"ahler--Einstein nonexistence results. Fourth, the remaining Fano indices
are identified by Kobayashi--Ochiai. Finally, Bando--Mabuchi uniqueness
identifies the metric, while projective rigidity identifies the realization
with the standard Veronese or Segre--Veronese model.

A notable feature of the argument is that the exclusion of the Fano-index-one
branch is independent of the ambient projective dimension. The codimension
enters only after the intrinsic classification, when the linearly full
realization is identified with the standard Veronese or Segre--Veronese
embedding.

\section{Preliminaries}
\label{sec:preliminaries}

\subsection{Normalization and the Ricci class}

We normalize the Fubini--Study metric so that its holomorphic sectional curvature is $4$.  Let
\[
  \varphi:(X^2,g)\longrightarrow(\PP^N,g_{\FS})
\]
be a holomorphic isometry and set
\[
  H:=\varphi^*\OO_{\PP^N}(1).
\]
We use the convention that, if $F$ is the diagonal of a positive Hermitian
kernel representing the pullback metric on $H$, then the associated K\"ahler
form is
\[
  \omega_g=\frac{i}{2}\,\partial\bar\partial\log F.
\]
The Ricci form is
\[
  \rho_g=-i\,\partial\bar\partial\log\det(g_{\alpha\bar\beta}),
\]
so that $c_1(H)=[\omega_g]/\pi$ and
$c_1(-K_X)=[\rho_g]/(2\pi)$.  Thus, if
\[
  \Ric(g)=\lambda g,
\]
then $\rho_g=\lambda\omega_g$ and
\begin{equation}\label{eq:ricci-class}
  c_1(-K_X)=\frac{\lambda}{2}\,c_1(H).
\end{equation}
In particular, intersecting \eqref{eq:ricci-class} with $H$ gives
\begin{equation}\label{eq:lambda-rational-intersection}
  \frac{\lambda}{2}
  =\frac{(-K_X)\cdot H}{H^2}\in\QQ,
\end{equation}
whenever $X$ is compact.

Let
\[
  \varphi:X^2\longrightarrow\PP^N
\]
be the given holomorphic immersion.  At each point $x\in X$, the embedded
projective tangent plane $T_x\varphi(X)\subset\PP^N$ determines a
$3$-dimensional vector subspace of $\CC^{N+1}$, namely the affine tangent
space to the cone over $\varphi(X)$ at any nonzero lift of $\varphi(x)$.
Thus the Gauss map is the holomorphic map
\[
  \gamma:X\longrightarrow\operatorname{Gr}(3,N+1),
  \qquad
  x\longmapsto \widehat T_xX,
\]
where $\widehat T_xX\subset\CC^{N+1}$ denotes this $3$-dimensional
vector subspace.  Composing $\gamma$ with the Pl\"ucker embedding
\[
  \operatorname{Gr}(3,N+1)
  \hookrightarrow
  \PP\!\left(\bigwedge^3\CC^{N+1}\right),
  \qquad
  W\longmapsto[\wedge^3 W],
\]
we obtain the projective Gauss map
\[
  \Gamma:X\longrightarrow
  \PP\!\left(\bigwedge^3\CC^{N+1}\right).
\]
The pullback of the hyperplane bundle under this map is
\begin{equation}\label{eq:gauss-line-bundle}
  \Gamma^*\OO(1)\simeq K_X\otimes H^3.
\end{equation}
The Gauss-map construction is classical.  In the Riemannian setting, a
generalized Gauss map for immersions into space forms was introduced by
Obata \cite{Obata1968}.  In the K\"ahler setting relevant here, Nishikawa
established the corresponding metric identity for K\"ahler immersions
\cite{Nishikawa1975}; see also Hulin
\cite[Proposition~3.1]{Hulin1996} for its formulation
and use in the study of projective K\"ahler--Einstein submanifolds.  With
our normalization, this identity specializes in complex dimension two to
\begin{equation}\label{eq:gauss-metric-general}
  \Gamma^*g_{\FS}=\left(3-\frac{\lambda}{2}\right)g.
\end{equation}

\subsection{Polarized Hermitian kernels}

We briefly recall the kernel language used below. If $L\to X$ is a holomorphic line bundle, we write
\[
L\boxtimes\bar L:=p_1^*L\otimes p_2^*\bar L
\]
for the external tensor product over $X\times\bar X$, where
$p_1:X\times\bar X\to X$ and $p_2:X\times\bar X\to\bar X$ are the projections.
A Hermitian section of $L\boxtimes\bar L$ can be written, after choosing a finite-dimensional subspace $U\subset H^0(X,L)$ with basis $s_1,\dots,s_r$, as
\begin{equation}\label{eq:kernel-representation}
  R(x,\bar y)=\sum_{i,j=1}^r C_{ij}s_i(x)\overline{s_j(y)},
\end{equation}
where $C=(C_{ij})$ is Hermitian.  The \emph{support} of $R$ is the smallest such subspace $U$, and the \emph{rank} of $R$ is the rank of the corresponding Hermitian form.  On the minimal support the matrix $C$ is nondegenerate.

For a full holomorphic map into projective space, the pullback of the standard Hermitian form gives such a kernel.  On the diagonal it is strictly positive, and the logarithm of the diagonal is a local Fubini--Study potential.  In what follows, a Hermitian kernel is said to have \emph{positive diagonal} if $R(x,\bar x)>0$ for all $x$.  This condition does not imply that its coefficient matrix on the minimal support is positive semidefinite; in particular, the forms used in the rank argument below may be indefinite.  Polarization is unique: an identity between real-analytic Hermitian kernels on the diagonal extends to $X\times\bar X$.

We shall repeatedly use the following elementary observation.

\begin{lemma}\label[lemma]{lem:constant-kernels}
Let $R_1$ and $R_2$ be Hermitian kernels with positive diagonal on the same holomorphic line bundle over a compact connected complex manifold.  If they induce the same K\"ahler form, then $R_1=cR_2$ for a positive constant $c$.
\end{lemma}

\begin{proof}
The quotient $R_1(x,\bar x)/R_2(x,\bar x)$ is a positive smooth function whose logarithm is pluriharmonic.  On a compact connected complex manifold every real pluriharmonic function is constant.  Polarization then gives the global identity.
\end{proof}

\begin{lemma}[Coprime root extraction]\label[lemma]{lem:coprime-root}
Let $Z$ be a compact connected complex manifold, let $L\to Z$ be a holomorphic line bundle, and let
\[
 F\in H^0(Z,L^q),\qquad G\in H^0(Z,L^r),
\]
be nonzero sections with $\gcd(q,r)=1$.  If
\[
 G^q=cF^r
\]
for some $c\in\CC^*$, then there exist $R\in H^0(Z,L)$ and constants
$a,b\in\CC^*$ such that
\[
 F=aR^q,\qquad G=bR^r.
\]
If, in addition, $Z=X\times\bar X$, $L=A\boxtimes\bar A$, and $F,G$ are
Hermitian kernels with positive diagonal, then $R$ can be chosen Hermitian
with positive diagonal.
\end{lemma}

\begin{proof}
Taking divisors gives
\[
 q\,\divisor(G)=r\,\divisor(F).
\]
Since $q$ and $r$ are coprime, there is an effective divisor $D$ on $Z$ such
that
\[
 \divisor(F)=qD,\qquad \divisor(G)=rD.
\]
Set $T:=\OO_Z(D)\otimes L^{-1}$.  Then $T^q\simeq\OO_Z$ and
$T^r\simeq\OO_Z$, hence B\'ezout's identity gives $T\simeq\OO_Z$.
Thus $\OO_Z(D)\simeq L$, and one may choose
$R\in H^0(Z,L)$ with $\divisor(R)=D$.  The quotients $F/R^q$ and $G/R^r$
are nowhere-vanishing holomorphic functions on compact $Z$, hence constants.

For the final assertion, let $\#$ denote the Hermitian involution on sections
of $A\boxtimes\bar A$, namely
\[
R^\#(x,\bar y):=\overline{R(y,\bar x)}.
\]  Since $F$ and $G$ are Hermitian, $D$ is
$\#$-invariant.  Hence $R^\#$ and $R$ have the same divisor and
$R^\#=\eta R$ for some $\eta\in\CC^*$.  Since $\#$ is antilinear, applying it
again gives $|\eta|^2=1$, hence $|\eta|=1$.  Multiplying $R$ by a suitable phase makes $R^\#=R$.  Its
diagonal is then real and nowhere zero.  Since $X$ is connected, it has
constant sign; replacing $R$ by $-R$ if necessary gives
$R(x,\bar x)>0$ for all $x$.
\end{proof}

\subsection{Standard facts on del Pezzo surfaces}

We shall use the standard classification and anticanonical geometry of smooth del Pezzo surfaces; see, for example, \cite[Chapter~8]{Dolgachev2012}.  Besides $\PP^1\times\PP^1$, every smooth del Pezzo surface is the blow-up of $\PP^2$ at $r\le8$ points in general position, with degree
\[
  d=(-K_X)^2=9-r.
\]
The Fano index $r_X$, defined by
\[
  -K_X=r_XA
\]
with $A$ primitive in $\operatorname{Pic}(X)$, belongs to $\{1,2,3\}$.  The index-three case is $\PP^2$, while the index-two case is the smooth quadric
\[
  Q^2\simeq\PP^1\times\PP^1;
\]
this is the surface case of the Kobayashi--Ochiai characterization \cite{KobayashiOchiai1973}.  All remaining smooth del Pezzo surfaces have index one.

For index one, $A=-K_X$.  We shall use the following familiar facts.  If $d=1$ or $2$, the anticanonical bundle is not very ample.  If $d=3$, the anticanonical model is a smooth cubic surface in $\PP^3$.  If $d=4$, the anticanonical model is a smooth complete intersection of two quadrics in $\PP^4$.  If $d=5$, the anticanonical linear system embeds the unique degree-five del Pezzo surface in $\PP^5$.

We also use Tian's classification of K\"ahler--Einstein del Pezzo surfaces \cite{Tian1990}: among smooth del Pezzo surfaces, the index-one surfaces of degrees seven and eight do not admit K\"ahler--Einstein metrics, whereas the surfaces of degrees at most six do.  Equivalently, the only smooth del Pezzo surfaces without a K\"ahler--Einstein metric are $\operatorname{Bl}_p\PP^2$ and $\operatorname{Bl}_{p_1,p_2}\PP^2$.

\begin{lemma}[The degree-five anticanonical metric]\label[lemma]{lem:Y5-not-Einstein}
Let $Y_5$ be the del Pezzo surface of degree five.  The $S_5$-invariant
Fubini--Study metric induced by the complete anticanonical embedding
\[
   Y_5\hookrightarrow\PP^5
\]
is not K\"ahler--Einstein.
\end{lemma}

\begin{proof}
The verification is an explicit curvature computation and is deferred to
Appendix~\ref{app:Y5}.
\end{proof}

\begin{proposition}\label[proposition]{prop:very-ample}
Let $X$ be a smooth del Pezzo surface of Fano index one carrying a
K\"ahler--Einstein metric $g$ induced by a holomorphic isometric immersion
$\varphi:(X,g)\to(\PP^N,g_{\FS})$ with
$H=\varphi^*\OO(1)=q(-K_X)$.  Then $-K_X$ is very ample.  In particular,
the conclusion is independent of $N$.
\end{proposition}

\begin{proof}
Set $A:=-K_X$.  Since $X$ is compact, Hulin's no-double-points theorem
\cite[main theorem, p.~278]{Hulin1996} shows that $\varphi$ is injective; hence
it is a holomorphic embedding.  Replacing the ambient projective space by the
projective span of the image, we may assume that $\varphi$ is linearly full.
The Ricci-class identity and $H=qA$ give
\[
c_1(A)=\frac{\lambda}{2}c_1(H)
      =\frac{\lambda q}{2}c_1(A),
\]
and therefore $\lambda=2/q$.

Let $F$ be the polarized Hermitian kernel of $\varphi$.  Since $H=qA$,
\[
   F\in H^0(X\times\bar X,A^q\boxtimes\bar A^{\,q}).
\]
Let $G$ be the polarized Hermitian kernel of the Pl\"ucker projective Gauss
map.  The classical Gauss-map metric identity recalled above gives
\[
   \Gamma^*g_{\FS}
   =\left(3-\frac{\lambda}{2}\right)g
   =\left(3-\frac1q\right)g
   =\frac{3q-1}{q}\,g.
\]
Equivalently,
\[
   \Gamma^*\OO(1)
   \simeq K_X\otimes H^3
   \simeq A^{-1}\otimes A^{3q}
   \simeq A^{3q-1},
\]
so
\[
   G\in H^0(X\times\bar X,
             A^{3q-1}\boxtimes\bar A^{\,3q-1}).
\]
The kernels $G^q$ and $F^{3q-1}$ are Hermitian sections of the same line
bundle
\[
   A^{q(3q-1)}\boxtimes\bar A^{\,q(3q-1)},
\]
and induce the same K\"ahler form.  By \cref{lem:constant-kernels}, there exists
$c>0$ such that
\[
   G^q=cF^{3q-1}.
\]

Put $Z=X\times\bar X$ and $L=A\boxtimes\bar A$.  Since
$\gcd(q,3q-1)=1$, \cref{lem:coprime-root} applied to
$G^q=cF^{3q-1}$ gives a Hermitian kernel
\[
   R\in H^0(X\times\bar X,A\boxtimes\bar A)
\]
with positive diagonal and constants $a,b>0$ such that
\[
   F=aR^q,
   \qquad
   G=bR^{3q-1}.
\]
This is the common anticanonical root of the projective and Gauss kernels.

\begin{remark}[Diastatic interpretation of the common root]
The common-root construction also has a natural interpretation in
Calabi's diastatic language.  Fix $p\in X$.  For a Hermitian kernel $K$
with positive diagonal, write
\[
   D_p^K(x)
   :=
   \log\frac{K(x,\bar x)K(p,\bar p)}{|K(x,\bar p)|^2}
\]
on its natural domain $K(x,\bar p)\ne0$.  On the common domain of the
pullback Fubini--Study diastases of $\varphi$ and of the projective
Gauss map $\Gamma$, the Gauss metric identity gives
\[
   qD_p^{\Gamma}=(3q-1)D_p^{\varphi}.
\]
The identities $F=aR^q$ and $G=bR^{3q-1}$ give at once
\[
   D_p^R
   =\frac1q D_p^{\varphi}
   =\frac1{3q-1}D_p^{\Gamma}.
\]
The additive constants cancel in the diastasis.  What the coprimality
$\gcd(q,3q-1)=1$ contributes is not this identity but the existence of $R$
itself as a global section of $A\boxtimes\bar A$, rather than merely a local
potential; it is that global, finite-rank object which the exclusions in
degrees three and four use.
Thus the anticanonical root may equivalently be viewed as a primitive
anticanonical diastasis.  We shall use the polarized kernel $R$ rather
than the diastasis in what follows, since the exclusions in degrees
three and four require the finite-rank Hermitian information carried
by its coefficient form.
\end{remark}

Let
\[
   U\subset H^0(X,A)
\]
be the minimal support of $R$.  Choose a basis of $U$ and write, with the first variable holomorphic and
the second antiholomorphic,
\[
   R(x,\bar y)=f(y)^*C f(x),
   \qquad
   R(x,\bar x)=f(x)^*C f(x),
\]
with $C$ nondegenerate Hermitian.  We stress that positivity of the diagonal
does not imply that $C$ is positive definite; the coefficient form may be
indefinite.  This is why the rank estimate below uses only nondegeneracy and
Sylvester's rank inequality.  Positivity on the diagonal implies that
$f(x)\ne0$ for every $x$, so $U$ is base-point-free.  Consider the multiplication map
\[
   \mu_q:\Sym^qU\longrightarrow H^0(X,qA).
\]
The nondegenerate Hermitian form $C$ on $U$ induces a Hermitian form on
$\Sym^qU$, and the Hermitian kernel $R^q$ is obtained from this form after
pushing it forward by $\mu_q$.  Consequently the minimal support of $R^q$ is
contained in $\operatorname{Im}\mu_q$.  Since $F=aR^q$, the minimal support of
$F$ is the same.  Because $\varphi$ is linearly full, that support is exactly
the span of the homogeneous coordinate sections defining $\varphi$.
Therefore every projective coordinate of $\varphi$ belongs to
$\operatorname{Im}\mu_q$ and is represented by a degree-$q$ homogeneous
polynomial in the coordinates defined by $U$.  Consequently, along $X$ the
map $\varphi$ factors projectively as
\[
   X\xrightarrow{\ \psi_U\ }\PP(U^*)
   \xrightarrow{\ \nu_q\ }\PP(\Sym^qU^*)
   \dashrightarrow\PP^N,
\]
where $\nu_q$ is the $q$-th Veronese embedding and the last arrow is the
projective linear map determined by the homogeneous coordinate sections of
$\varphi$.  Its center is disjoint from
$\nu_q(\psi_U(X))$, because $\varphi$ is everywhere defined, so this is an
honest holomorphic factorization on $\psi_U(X)$.

It follows immediately that
\[
   \psi_U(x)=\psi_U(y)\Longrightarrow \varphi(x)=\varphi(y),
   \qquad
   d\psi_U(v)=0\Longrightarrow d\varphi(v)=0.
\]
Since $\varphi$ is an embedding, $\psi_U$ separates points and tangent
vectors.  Hence
\[
   \psi_U:X\hookrightarrow\PP(U^*)
\]
is an embedding.  In particular,
\[
   A=-K_X
\]
is very ample.
\end{proof}

\section{Proof of the main theorem}
\label{sec:proof-main}

We now prove \cref{thm:main} in a single sequence of steps.  The proof is
organized so that the standard algebraic-geometric input and the genuinely
metric restrictions imposed by projective inducibility remain clearly
separated.

\begin{proof}[Proof of \cref{thm:main}]
Since $X$ is compact, it is complete.  By Hulin's no-double-points theorem
\cite[main theorem, p.~278]{Hulin1996}, the given K\"ahler immersion
$\varphi$ is injective; since
it is an immersion from a compact manifold into a Hausdorff manifold, it is a
holomorphic embedding.
We shall use both point separation and tangent-vector separation by $\varphi$
throughout the proof.

\medskip
\noindent
\textbf{Step 1: Fano reduction.}
By Hulin's positivity theorem \cite[Theorem]{Hulin2000}, the Einstein constant is strictly positive:
\[
   \lambda>0.
\]
Since
\[
   c_1(-K_X)=\frac{\lambda}{2}\,c_1(H),
   \qquad H:=\varphi^*\OO_{\PP^N}(1),
\]
the anticanonical bundle is positive, hence ample.  Therefore $X$ is a smooth
del Pezzo surface.

\medskip
\noindent
\textbf{Step 2: Reduction to the three Fano indices.}
Write
\[
   -K_X=r_XA,
\]
with $A$ primitive and ample.  A smooth del Pezzo surface is rational, hence
$H^1(X,\OO_X)=0$ and $\operatorname{Pic}^0(X)=0$.  Thus the first Chern class
map
\[
   c_1:\operatorname{Pic}(X)\longrightarrow H^2(X,\mathbb Z)
\]
is injective and identifies $\operatorname{Pic}(X)$ with the lattice
$c_1(\operatorname{Pic}(X))$.  Since $A$ is primitive in
$\operatorname{Pic}(X)$, the class $c_1(A)$ is primitive in this lattice.
The Ricci-class identity places $c_1(H)$ on the positive ray spanned by
$c_1(A)$; hence
\[
   c_1(H)=m c_1(A)
\]
for a unique integer $m\ge1$.  Injectivity of $c_1$ then gives
\[
   H=mA.
\]
Therefore
\[
   \frac{\lambda}{2}=\frac{r_X}{m}.
\]
For a smooth del Pezzo surface $r_X\in\{1,2,3\}$.  It remains to exclude the
index-one branch and then identify the cases $r_X=2,3$.

\medskip
\noindent
\textbf{Step 3: Exclusion of Fano index one.}
Assume
\[
   r_X=1.
\]
Then $A=-K_X$.  To avoid confusion with the polarization integer $m$
used in the statement of the theorem, we denote the index-one polarization
multiple by $q$; thus
\[
   H=qA,
   \qquad q\in\mathbb Z_{>0},
\]
and
\[
   \lambda=\frac{2}{q}.
\]
After replacing the ambient projective space by the projective span of the
image, we may assume that $\varphi$ is linearly full.

\smallskip
\noindent
\emph{Step 3a: degrees one and two.}
By \cref{prop:very-ample}, the anticanonical bundle $A=-K_X$ is very ample.
For a del Pezzo surface of degree $1$ or $2$, however, the anticanonical
bundle is not very ample.  Hence degrees $1$ and $2$ cannot occur.

\smallskip
\noindent
\emph{Step 3b: degrees three and four.}
For the remaining degrees, retain the common-root construction from the proof
of \cref{prop:very-ample}.  Thus the polarized projective kernel $F$ admits a
Hermitian anticanonical root
\[
   R\in H^0(X\times\bar X,A\boxtimes\bar A),
   \qquad F=aR^q,
\]
with positive diagonal.  Let $U\subset H^0(X,A)$ be its minimal support; as
shown in that proof, the associated map
\[
   \psi_U:X\hookrightarrow\PP(U^*)
\]
is an embedding.

Set
\[
   g_A:=q^{-1}g.
\]
Since $F=aR^q$, the root kernel induces precisely this rescaled metric:
\[
   \omega_A:=\omega_{g_A}
   =\frac{i}{2}\partial\bar\partial\log R.
\]
Since constant rescaling does not change the Ricci tensor,
\[
   \Ric(g_A)=\Ric(g)=\frac2q g=2g_A.
\]
Locally choose a holomorphic frame of $A$ and write
\[
   R(z,\bar z)=f(z)^*Cf(z),
\]
where $C$ is nondegenerate on $U$.  Put
\[
   r:=\dim U,
   \qquad
   f_0=f,
   \qquad
   f_i=\frac{\partial f}{\partial z_i},\ i=1,2.
\]
With the convention fixed in \cref{sec:preliminaries}, the local matrix of
$g_A$ is
\[
   (g_A)_{i\bar j}=\frac{\partial^2}{\partial z_i\partial\bar z_j}\log R.
\]
The Schur-complement identity therefore gives the exact diagonal identity
\begin{equation}\label{eq:gram-schur}
   \det\bigl(f_a^*Cf_b\bigr)_{0\le a,b\le2}
   =R^3\det\bigl((g_A)_{i\bar j}\bigr).
\end{equation}
The global meaning of this determinant is most transparent through first
jets.  The first-jet sequence
\[
   0\longrightarrow\Omega_X^1\otimes A
   \longrightarrow J^1A
   \longrightarrow A
   \longrightarrow0
\]
has rank $3$, and the evaluation of first jets gives a canonical map
\[
   U\otimes\OO_X\longrightarrow J^1A.
\]
Taking determinants yields the Wronskian map
\[
   \mathcal J_U:\bigwedge\nolimits^3U
   \longrightarrow H^0\!\left(X,\det J^1A\right)
   =H^0\!\left(X,A^3\otimes K_X\right)
   =H^0(X,2A),
\]
where the last equality uses $K_X=-A$.  In a local frame this map is
represented exactly by $f\wedge f_1\wedge f_2$.

The nondegenerate Hermitian form $C$ on $U$ induces a Hermitian form on
$\bigwedge^3U$.  Pushing it forward through $\mathcal J_U$ defines a global
Hermitian kernel $W$ on $2A$; its diagonal is the determinant on the
left-hand side of \eqref{eq:gram-schur}.  Consequently its support is
contained in $\operatorname{Im}\mathcal J_U$, and hence
\[
   \rank(W)\le\dim\operatorname{Im}\mathcal J_U
   \le\dim\bigwedge\nolimits^3U=\binom r3.
\]
On the diagonal, \eqref{eq:gram-schur} is positive because $R>0$ and $g_A$ is
positive definite.  Its induced K\"ahler form is
\begin{align*}
   \omega_W
   &=\frac{i}{2}\partial\bar\partial\log W\\
   &=3\omega_A-\frac12\rho_{g_A}.
\end{align*}
Since $\Ric(g_A)=2g_A$, equivalently $\rho_{g_A}=2\omega_A$, we obtain
\[
   \omega_W=2\omega_A.
\]
But $R^2$ is a Hermitian kernel with positive diagonal on $2A$ and induces the same form
$2\omega_A$.  By \cref{lem:constant-kernels},
\[
   W=cR^2
\]
for some $c>0$.  Consequently
\[
   \rank(R^2)=\rank(W)\le\binom r3.
\]

For the opposite inequality, let
\[
   S:=\dim\Sym^2U=\binom{r+1}{2}
\]
and let
\[
   \mu:\Sym^2U\longrightarrow H^0(X,2A)
\]
be multiplication.  Write
\[
   h:=\rank\mu.
\]
The nondegenerate Hermitian form defined by $R$ induces a nondegenerate
Hermitian form on $E:=\Sym^2U$, whose dimension is $S$.  Choose a basis of
$\operatorname{Im}\mu$ and let $M$ be the $h\times S$ matrix of the
surjection $E\to\operatorname{Im}\mu$.  If $B$ is the nonsingular Hermitian
matrix induced by $R^2$ on $E$, then the coefficient matrix of the
push-forward kernel on $\operatorname{Im}\mu$ is $MBM^*$.  Since
$\rank M=h$ and $B$ is invertible, Sylvester's rank inequality yields
\[
   \rank(R^2)=\rank(MBM^*)
   \ge \rank(MB)+\rank(M^*)-S
   =2h-S.
\]
Thus
\[
   2h-\binom{r+1}{2}
   \le \rank(R^2)
   \le \binom r3.
\]

For $d=3$, the standard anticanonical model is a smooth cubic surface in
$\PP^3$ \cite[Chapter~8]{Dolgachev2012}.  Since $h^0(X,A)=4$ and $U$ embeds
$X$, one has $r=4$.  Indeed $r\ge3$, since $\psi_U$ embeds the surface $X$ in
$\PP(U^*)=\PP^{r-1}$; and $r\ne3$, for if $r=3$, then
$\psi_U:X\to\PP^2$ would be a closed embedding of a projective surface,
hence its image would be all of $\PP^2$ and $X\simeq\PP^2$, contradicting
$r_X=1$.  Since also $r\le h^0(X,A)=4$, it follows that $r=4$.  The cubic
is contained in no quadric, hence
\[
   S=\binom52=10,
   \qquad h=10,
\]
and therefore
\[
   10\le\rank(R^2)\le\binom43=4,
\]
a contradiction.

For $d=4$, the anticanonical model is a smooth complete intersection of two
quadrics in $\PP^4$ \cite[Chapter~8]{Dolgachev2012}.  Here $h^0(X,A)=5$ and
$r\ge4$ by the preceding argument.  In fact $r=5$: if $r=4$, the embedding
defined by $U$ would realize $X$ as a surface
in $\PP^3$.  Since the embedding is defined by $A$ and $A^2=4$, its image
has degree four; being a surface in $\PP^3$, it is therefore a quartic
hypersurface.  Its canonical bundle is trivial by adjunction, contrary to
$K_X=-A$.  Thus
\[
   S=\binom62=15,
   \qquad h=15-2=13,
\]
and
\[
   11=2\cdot13-15
   \le\rank(R^2)
   \le\binom53=10,
\]
again a contradiction.

\smallskip
\noindent
\emph{Step 3c: degree five.}
Let
\[
   Y_5\simeq\operatorname{Bl}_{p_1,\dots,p_4}\PP^2
\]
be the degree-five del Pezzo surface.  It is unique up to isomorphism, and
\[
   \Aut(Y_5)\simeq S_5.
\]
Moreover, the anticanonical section space
$H^0(Y_5,-K_{Y_5})$ is an irreducible six-dimensional complex representation
of $S_5$ \cite[Theorem~2.6]{BauerCatanese2021}.

Assume that a projectively induced K\"ahler--Einstein metric exists with
polarization $q(-K_{Y_5})$.  The rescaled metric
\[
   g_A:=q^{-1}g
\]
is K\"ahler--Einstein in the anticanonical class.  For every
$\sigma\in\Aut(Y_5)$, the pullback $\sigma^*g_A$ is another
K\"ahler--Einstein metric in the same class.  To place the uniqueness theorem
in its standard anticanonical normalization, set
\[
   \widehat g_A:=2g_A.
\]
Then
\[
   [\omega_{\widehat g_A}]=2\pi c_1(-K_{Y_5}),
   \qquad
   \rho_{\widehat g_A}=\omega_{\widehat g_A}.
\]
Since $\Aut^0(Y_5)$ is trivial, Bando--Mabuchi uniqueness
\cite{BandoMabuchi1987}, applied to $\widehat g_A$ and
$\sigma^*\widehat g_A$, gives
\[
   \sigma^*\widehat g_A=\widehat g_A.
\]
Dividing by $2$ yields
\[
   \sigma^*g_A=g_A
\]
for every $\sigma\in S_5$.

Let $R$ be the anticanonical root constructed above.  We use the
canonical $\Aut(Y_5)$-linearization of $-K_{Y_5}$, and hence the induced
actions on
\[
V:=H^0(Y_5,-K_{Y_5})
\qquad\text{and}\qquad
V^*.
\]
For every $\sigma\in\Aut(Y_5)$, the kernels $\sigma^*R$ and $R$ define
Hermitian metrics on $-K_{Y_5}$ with the same curvature, because
$\sigma^*g_A=g_A$.  By \cref{lem:constant-kernels},
\[
\sigma^*R=c_\sigma R
\]
for some $c_\sigma>0$.  Since every $\sigma\in S_5$ has finite order,
$c_\sigma=1$.  Thus $R$ is $S_5$-invariant.

Let
\[
U\subset V
\]
be the minimal support of $R$.  The invariance of $R$ implies that $U$ is an
$S_5$-subrepresentation of $V$.  Since $R\neq0$, one has $U\neq0$, and since
$V$ is an irreducible six-dimensional $S_5$-representation, it follows that
\[
U=V.
\]

We now make explicit the Hermitian-space convention.  Choose a basis
\[
\mathcal B=(s_1,\ldots,s_6)
\]
of $V$.  In a local frame of $-K_{Y_5}$, write
\[
f(x)=
\begin{pmatrix}
s_1(x)\\
\vdots\\
s_6(x)
\end{pmatrix}.
\]
The vector $f(x)$ is naturally the coordinate vector, in the dual basis
$\mathcal B^*$, of the evaluation functional
\[
\operatorname{ev}_x\in V^*.
\]
Accordingly, a polarized kernel
\[
R(x,\bar y)=f(y)^*C f(x)
\]
is intrinsically determined by a nondegenerate Hermitian form on $V^*$; the
matrix $C$ is its matrix in the dual basis $\mathcal B^*$.

Since $R$ is $S_5$-invariant, this Hermitian form on $V^*$ is
$S_5$-invariant.  The dual representation $V^*$ is irreducible, and hence, by
Schur's lemma, the coefficient form of $R$ is a real scalar multiple of the
unique $S_5$-invariant positive Hermitian form on $V^*$.  The scalar is
nonzero because the minimal support of $R$ is all of $V$, and it is positive
because
\[
R(x,\bar x)>0
\qquad\text{for every }x\in Y_5.
\]
Thus, up to an irrelevant positive constant, $R$ is the Fubini--Study kernel
associated with the unique $S_5$-invariant Hermitian structure on the complete
anticanonical system.

Equivalently, let $\langle\cdot,\cdot\rangle_0$ denote the unique
$S_5$-invariant positive Hermitian inner product on the section space $V$, up
to a positive scalar, and let $H_0$ be its Gram matrix in the basis
$\mathcal B$.  The induced dual Hermitian form on $V^*$ has matrix
\[
H_0^{-1}
\]
in the dual basis $\mathcal B^*$.  Therefore the coefficient matrix of the
corresponding Fubini--Study kernel is, up to a positive scalar,
\[
C\sim H_0^{-1}.
\]
This is precisely the convention used in Appendix~\ref{app:Y5}, where the
$S_5$-invariant Hermitian product on $V$ is represented by $H_0$ and the
associated Fubini--Study kernel is represented, up to scale, by
\[
G=10H_0^{-1}.
\]

Consequently, the metric $g_A$ induced by the primitive root $R$ is exactly
the distinguished $S_5$-invariant Fubini--Study metric associated with the
complete anticanonical embedding
\[
Y_5\hookrightarrow\PP(V^*)\simeq\PP^5.
\]
By \cref{lem:Y5-not-Einstein}, this metric is not K\"ahler--Einstein.  This
contradicts the construction of $g_A$ and therefore excludes the degree-five
case.

\smallskip
\noindent
\emph{Step 3d: degrees six, seven and eight.}
The degree-six del Pezzo surface is the toric blow-up of $\PP^2$ at three
non-collinear points.  By the toric rigidity theorem of Arezzo--Loi--Zuddas
\cite[Proposition~4.2]{ArezzoLoiZuddas2012}, a smooth compact toric manifold of
complex dimension at most four carrying a projectively induced K\"ahler--Einstein metric must be
a projective space or a product of projective spaces.  The degree-six del
Pezzo surface is neither, so it is excluded.  Alternatively, since the
degree-six del Pezzo surface is toric, it is excluded by the toric
classifications of Manno and Salis
\cite{MannoSalis2022,MannoSalis2026}.

The index-one del Pezzo surfaces of degrees seven and eight are respectively
$\operatorname{Bl}_{p_1,p_2}\PP^2$ and $\operatorname{Bl}_p\PP^2$.
By Tian's classification they do not admit K\"ahler--Einstein metrics, hence
they are excluded a fortiori.

We have now excluded every index-one del Pezzo surface.  Therefore
\[
   r_X\ne1.
\]

\medskip
\noindent
\textbf{Step 4: The remaining Fano indices.}
We now return to the polarization integer $m$ of Step~2; the symbol $q$ was
used only within the excluded index-one branch.  Only
\[
   r_X=2\quad\text{or}\quad r_X=3
\]
remain.  By the Kobayashi--Ochiai characterization, if $r_X=3$ then
\[
   X\simeq\PP^2,
   \qquad
   A=\OO_{\PP^2}(1),
\]
whereas if $r_X=2$ then
\[
   X\simeq Q^2\simeq\PP^1\times\PP^1,
   \qquad
   A=\OO(1,1).
\]
Thus the underlying complex surface is already one of the two homogeneous
models.

\medskip
\noindent
\textbf{Step 5: Identification of the metric and the projective realization.}
Suppose first that $X\simeq\PP^2$.  Since $H=mA$,
\[
   H=\OO_{\PP^2}(m).
\]
Let
\[
   g_0:=m g_{\FS}.
\]
Both $g$ and $g_0$ have Einstein constant $6/m$.  Rescale them to
\[
   \widetilde g:=\frac6m\,g,
   \qquad
   \widetilde g_0:=\frac6m\,g_0=6g_{\FS}.
\]
Then both lie in $2\pi c_1(-K_{\PP^2})$ and satisfy
$\rho=\omega$.  Bando--Mabuchi uniqueness \cite{BandoMabuchi1987} gives an
automorphism $F$ of $\PP^2$ such that
\[
   F^*\widetilde g_0=\widetilde g.
\]
Scaling back gives $F^*g_0=g$.  After applying this biholomorphic isometry
of the source, we may therefore assume
\[
   g=m g_{\FS}.
\]
The standard $m$-th Veronese map
\[
   \nu_m:(\PP^2,m g_{\FS})
   \longrightarrow
   \PP^{\binom{m+2}{2}-1}
\]
is a full holomorphic isometry when the degree-$m$ monomials are taken with
the multinomial weights $\binom{m}{\alpha}^{1/2}$; with these weights
$\nu_m^*g_{\FS}=m g_{\FS}$.  After restricting the original target to
the projective span of $\varphi(X)$, the original map is also full.  By
Calabi's rigidity theorem \cite{Calabi1953}, the two full projective
realizations differ by a projective-unitary transformation, and hence
\[
   N=\binom{m+2}{2}-1
\]
in the linearly full realization.

If instead $X\simeq\PP^1\times\PP^1$, then
\[
   H=\OO(m,m).
\]
Let
\[
   g_0:=m(g_{\FS}\oplus g_{\FS}).
\]
Both $g$ and $g_0$ have Einstein constant $4/m$.  Hence
\[
   \widetilde g:=\frac4m\,g,
   \qquad
   \widetilde g_0:=\frac4m\,g_0
   =4(g_{\FS}\oplus g_{\FS})
\]
are K\"ahler--Einstein metrics in the normalized anticanonical class
$2\pi c_1(-K_X)$ with $\rho=\omega$.  By Bando--Mabuchi uniqueness
\cite{BandoMabuchi1987}, an automorphism $F$ of
$\PP^1\times\PP^1$ satisfies
\[
   F^*\widetilde g_0=\widetilde g.
\]
Scaling back gives $F^*g_0=g$.  After applying this source automorphism, we
may therefore assume
\[
   g=m(g_{\FS}\oplus g_{\FS}).
\]
The complete linear system $|\OO(m,m)|$, with the standard binomial
weights on each factor, gives the full Segre--Veronese holomorphic isometry
\[
   \PP^1\times\PP^1
   \longrightarrow
   \PP^{(m+1)^2-1},
\]
whose pullback metric is $m(g_{\FS}\oplus g_{\FS})$.  By Calabi's rigidity theorem \cite{Calabi1953}, the original full
realization differs from the standard one by a projective-unitary
transformation.  Hence
\[
   N=(m+1)^2-1.
\]
Both metrics are homogeneous.  This proves the theorem.
\end{proof}

\begin{corollary}[Homogeneity in complex dimension two]\label[corollary]{cor:homogeneity}
Every connected compact projectively induced K\"ahler--Einstein surface is homogeneous.
Equivalently, \Cref{conj:homogeneity} holds in complex dimension two.
\end{corollary}

\appendix

\section{The degree-five anticanonical metric: explicit curvature calculation}
\label{app:Y5}

All matrix operations and differentiations in this appendix are carried out
in exact integer or rational arithmetic; no floating-point approximation or
numerical sampling is used.

Let
\[
Y=\operatorname{Bl}_{p_1,\dots,p_4}\PP^2,
\qquad
(-K_Y)^2=5,
\]
with
\[
p_1=(1:0:0),\qquad
p_2=(0:1:0),\qquad
p_3=(0:0:1),\qquad
p_4=(1:1:1).
\]

Using the Bauer--Catanese basis \cite[Proposition~2.2]{BauerCatanese2021}
\[
s_{ij}=x_ix_j(x_j-x_k),
\qquad
\{i,j,k\}=\{1,2,3\},
\]
ordered as
\[
\mathcal B=(s_{12},s_{13},s_{21},s_{23},s_{31},s_{32}),
\]
the $S_5$-invariant Hermitian product is unique up to scale because the
six-dimensional anticanonical representation is irreducible.

For completeness, we verify an invariant Gram matrix directly.  Bauer and
Catanese identify the space $V'$ of cubics through $p_1,\dots,p_4$ with
$V=H^0(Y,-K_Y)$ by $s\mapsto s\,(dx_1\wedge dx_2\wedge dx_3)^{-1}$, so that
an odd linear coordinate permutation acquires its determinant sign; for the
subgroup $S_3$ this is recorded as
$\tau(s_{ij})=\epsilon(\tau)\,s_{\tau(i)\tau(j)}$
\cite[Lemma~2.4]{BauerCatanese2021}, and the action of a $5$-cycle on the
basis $\mathcal B$ is computed in the proof of
\cite[Theorem~2.6]{BauerCatanese2021}.  With the convention that columns are
the coordinates of the images of the basis vectors, this gives the following
matrices of the anticanonical representation for the generators
\[
\tau=(1\,2),
\qquad
\sigma=(1\,2\,3\,4\,5)
\]
of $S_5$:
\[
P_\tau=
\begin{pmatrix*}[r]
 0 &  0 & -1 &  0 &  0 &  0 \\
 0 &  0 &  0 & -1 &  0 &  0 \\
-1 &  0 &  0 &  0 &  0 &  0 \\
 0 & -1 &  0 &  0 &  0 &  0 \\
 0 &  0 &  0 &  0 &  0 & -1 \\
 0 &  0 &  0 &  0 & -1 &  0
\end{pmatrix*},
\]
and
\[
P_\sigma=
\begin{pmatrix*}[r]
0 &  0 &  0 &  1 &  0 &  1 \\
0 &  1 &  0 &  0 &  1 &  0 \\
0 &  0 &  1 & -1 &  1 & -1 \\
0 & -1 &  0 &  0 &  0 &  0 \\
1 & -1 &  0 &  0 & -1 &  1 \\
0 &  0 &  0 & -1 &  0 &  0
\end{pmatrix*}.
\]

Consider the Hermitian matrix
\[
H_0=
\begin{pmatrix*}[r]
 5 & -2 &  2 & -1 & -1 &  2 \\
-2 &  5 & -1 &  2 &  2 & -1 \\
 2 & -1 &  5 & -2 &  2 & -1 \\
-1 &  2 & -2 &  5 & -1 &  2 \\
-1 &  2 &  2 & -1 &  5 & -2 \\
 2 & -1 & -1 &  2 & -2 &  5
\end{pmatrix*}.
\]
The eigenvalues of $H_0$ are
\[
1,\ 2,\ 2,\ 5,\ 10,\ 10,
\]
so $H_0$ is positive definite.  It satisfies
\[
P_\tau^*H_0P_\tau=H_0,
\qquad
P_\sigma^*H_0P_\sigma=H_0.
\]
Since $\tau$ and $\sigma$ generate $S_5$, this verifies directly the
$S_5$-invariance of $H_0$. By irreducibility, Schur's lemma shows that
every $S_5$-invariant Hermitian form is a real scalar multiple of $H_0$.
Thus $H_0$ defines, up to positive scale, the unique $S_5$-invariant
Hermitian product on the section space.

For this Hermitian product, the corresponding Fubini--Study kernel is
represented in the chosen basis by the inverse Hermitian matrix. Hence, up to
an irrelevant positive scalar,
\[
G=10H_0^{-1}
=
\begin{pmatrix*}[r]
 4 &  0 & -2 &  1 &  1 & -2 \\
 0 &  4 &  1 & -2 & -2 &  1 \\
-2 &  1 &  4 &  0 & -2 &  1 \\
 1 & -2 &  0 &  4 &  1 & -2 \\
 1 & -2 & -2 &  1 &  4 &  0 \\
-2 &  1 &  1 & -2 &  0 &  4
\end{pmatrix*}.
\]
Its eigenvalues are
\[
1,\ 1,\ 2,\ 5,\ 5,\ 10,
\]
so $G$ is positive definite.

On the affine chart $x_3=1$, with
\[
z=x_1,\qquad w=x_2,
\]
the section vector is
\[
s(z,w)=
\begin{pmatrix}
zw(w-1)\\
z(1-w)\\
zw(z-1)\\
w(1-z)\\
z(z-w)\\
w(w-z)
\end{pmatrix}.
\]
The local Fubini--Study potential is
\[
\Phi_{\mathrm{loc}}
=
\log\!\bigl(s(z,w)^*G\,s(z,w)\bigr).
\]
Thus
\[
g_{i\bar j}
=
\frac{\partial^2\Phi_{\mathrm{loc}}}
     {\partial z_i\partial\bar z_j},
\qquad
\Ric_{i\bar j}
=
-\frac{\partial^2}
       {\partial z_i\partial\bar z_j}
\log\det(g_{k\bar\ell}),
\]
where
\[
(z_1,z_2)=(z,w).
\]
With the conventions of \cref{sec:preliminaries}, a metric satisfying
$\Ric(g)=\lambda g$ there has
$\Ric_{i\bar j}=\tfrac{\lambda}{2}g_{i\bar j}$ in the notation just
fixed.  For the complete anticanonical polarization the expected Einstein
ratio in these coefficients is therefore $1$.

Write $N:=s^*Gs$ and
\[
D:=\det\bigl(u_a^*Gu_b\bigr)_{0\le a,b\le2},
\qquad
u_0=s,\quad u_1=\partial_zs,\quad u_2=\partial_ws.
\]
The Schur-complement identity \eqref{eq:gram-schur} gives
$\det(g_{i\bar j})=D/N^3$, hence
\[
\Ric_{i\bar j}
=3g_{i\bar j}
-\frac{\partial^2\log D}{\partial z_i\partial\bar z_j}.
\]
At $(z,w)=(2,0)$ one has
\[
s=(0,2,0,0,4,0),\qquad N=48,\qquad D=5184,
\qquad \det(g_{i\bar j})=\frac{3}{64}.
\]
Exact rational arithmetic then gives
\[
g_{i\bar j}
=
\begin{pmatrix*}[r]
\dfrac{1}{12} & \dfrac{1}{24}\\[6pt]
\dfrac{1}{24} & \dfrac{7}{12}
\end{pmatrix*},
\qquad
\Ric_{i\bar j}
=
\begin{pmatrix*}[r]
\dfrac{7}{108} & \dfrac{7}{216}\\[6pt]
\dfrac{7}{216} & \dfrac{67}{324}
\end{pmatrix*}.
\]

If the metric were Einstein, these two Hermitian matrices would be
proportional. However,
\[
\frac{\Ric_{1\bar1}}{g_{1\bar1}}
=
\frac79,
\qquad
\frac{\Ric_{2\bar2}}{g_{2\bar2}}
=
\frac{67}{189},
\]
and these numbers are different. Hence the unique $S_5$-compatible
Fubini--Study metric induced by the complete anticanonical embedding is not
Einstein. This proves \cref{lem:Y5-not-Einstein}.

\section*{Acknowledgments}

The second author has been partially supported by GNSAGA of INdAM
and by ProBiki of Fondazione di Sardegna.

\end{document}